\documentclass{amsart}
\usepackage{amsmath,amssymb}
\usepackage{graphicx,subfigure}

\usepackage{color}

\newtheorem{theorem}{Theorem}[section]
\newtheorem{proposition}[theorem]{Proposition}
\newtheorem{corollary}[theorem]{Corollary}

\theoremstyle{definition}
 \newtheorem{definition}[theorem]{Definition}

\newtheorem*{definition*}{Definition}

\theoremstyle{remark}
\newtheorem{remark}[theorem]{Remark}

\numberwithin{equation}{section}

\newcommand{\al}{\alpha}
\newcommand{\be}{\beta}
\newcommand{\de}{\delta}
\newcommand{\ep}{\varepsilon}

\newcommand{\ga}{\gamma}

\newcommand{\vp}{\varphi}

\newcommand{\De}{\Delta}

\newcommand{\Si}{\Sigma}
\newcommand{\Om}{\Omega}

\def\RR{\mathbb{R}}

\def\ZZ{\mathbb{Z}}

\def\TT{\mathbb{T}}

\newcommand{\cD}{{\mathcal D}}

\newcommand{\cH}{{\mathcal H}}

\newcommand{\cK}{{\mathcal K}}

\newcommand{\cR}{{\mathcal R}}

\newcommand\cZ{\mathcal Z}

\newcommand{\pd}{\partial}
\newcommand\minus\backslash

\newcommand\lan\langle
\newcommand\ran\rangle

\DeclareMathOperator\Div{div}

\def\Harm{\cH_\Om }
\def\Harmt{\cH_{\Om^t} }

\renewcommand\leq\leqslant
\renewcommand\geq\geqslant
\newlength{\intwidth}

\newcommand\loc{_{\mathrm{loc}}}

\newcommand\BOm{\overline\Om}

 \DeclareMathOperator\curl{curl}

\newcommand\restr{\!\upharpoonright_{\pd\Om}}\newcommand\restrt{\!\upharpoonright_{\pd\Om^t}}
\newcommand\restrr{\!\upharpoonright_{\pd D}}

\newcommand\Dom{\cD_\Om}

\DeclareMathOperator\BS{BS}

\begin{document}

\title[Non-existence of axisymmetric optimal domains]{Non-existence
  of axisymmetric optimal domains with smooth boundary for the first curl eigenvalue}

\author{Alberto Enciso}
\address{Instituto de Ciencias Matem\'aticas, Consejo Superior de
  Investigaciones Cient\'\i ficas, 28049 Madrid, Spain}
\email{aenciso@icmat.es, dperalta@icmat.es}

\author{Daniel Peralta-Salas}

%
%
\begin{abstract}
  We say that a bounded domain~$\Om$ is optimal for the first positive
  curl eigenvalue $\mu_1(\Om)$ if $\mu_1(\Om)\leq \mu_1(\Om')$ for any
  domain~$\Om'$ with the same volume. In spite of the fact that
  $\mu_1(\Om)$ is uniformly lower bounded in terms of the volume, in
  this paper we prove that there are no axisymmetric optimal (and even
  locally minimizing) domains with $C^{2,\al}$ boundary that
  satisfies a mild technical assumption. As a particular case, this
  rules out the existence of $C^{2,\al}$~optimal axisymmetric domains with a convex
  section. An analogous result holds in the case of the first negative
  curl eigenvalue.
\end{abstract}
\maketitle

\section{Introduction}

Given a bounded domain $\Om\subset\RR^3$, a classical result of Giga and Yoshida~\cite{Giga} (see also~\cite{Hiptmair}) states that curl defines a self-adjoint
operator on $\Om$ with compact resolvent whose domain~$\Dom$ is dense in the space
\[
\cK(\Om)=\Big\{ v\in L^2(\Om): \Div v=0\,,\; v\restr\cdot
N=0\,,\; \int_\Om
v\cdot h\, dx=0 \; \text{for all } h\in\Harm \Big\}\,.
\]
Here $\Harm $ denotes the space of harmonic fields on~$\Om$
that are tangent to the boundary, and $N$ is the outward-pointing normal to the boundary (of course, $v\restr\cdot
N=0$ has to be understood in the sense of traces). In this work we will only consider
domains which are smooth enough, e.g., with a $C^{2,\al}$ boundary $\pd\Om:=\overline \Om\backslash \Om$.

The eigenfunctions of curl are then vector fields on~$\Om$ that satisfy
\begin{align}\label{P}
\curl u_k=\mu_k(\Om)\, u_k\quad\text{in }\Om\,,
\end{align}
and belong to $\Dom$. It is well known that there are infinitely many positive and negative
eigenvalues $\{\mu_k(\Om)\}_{k=-\infty}^\infty$ of curl, which tend to
$\pm\infty$  as
$k\to\pm\infty$ and which one can label so that
\[
\cdots \leq \mu_{-3}(\Om)\leq \mu_{-2}(\Om)\leq \mu_{-1}(\Om) <0<\mu_1(\Om)\leq \mu_2(\Om)\leq \mu_3(\Om)\leq \cdots
\]
We will refer to~$\mu_1(\Om)$ and $\mu_{-1}(\Om)$ as the {\em first
  positive eigenvalue}\/ and {\em first negative eigenvalue}\/ of the curl
operator, respectively. Notice that their multiplicities can be higher than~1. One should recall that, when considered in absolute value, the first curl eigenvalue admits a variational
formulation: indeed,
\[
\min\{\mu_{-1}^2(\Om),\mu_1^2(\Om)\}= \inf_{v\in \Dom\backslash \{0\}} \frac{\int_\Om|\curl v|^2\, dx}{\int_\Om |v|^2\, dx}\,.
\]

In this work we are interested in domains that minimize the first (positive or negative) curl eigenvalue among any other domain with the same volume. More precisely, we introduce the following:

\begin{definition}
A $C^{2,\al}$ bounded domain $\Om$ is {\em optimal}\/ (respectively,
{\em locally optimal}\/) for the first
positive
curl eigenvalue if
\[
\mu_1(\Om)\leq \mu_1(\Om')
\]
for any $C^{2,\al}$ domain~$\Om'$ of the same volume (respectively, for
any $C^{2,\al}$-small perturbation~$\Om'$ of~$\Om$ with the same
volume). Optimal and locally optimal domains for the first negative
curl eigenvalue are defined analogously: $|\mu_{-1}(\Om)|\leq |\mu_{-1}(\Om')|$.
\end{definition}

The analysis of optimal domains, including questions of existence,
uniqueness and regularity, is
a classical subject in spectral theory. For the Dirichlet Laplacian,
the Faber--Krahn inequality implies that the ball is the only optimal
domain for the first eigenvalue. Even in the case of higher eigenvalues
of the Dirichlet Laplacian, the situation is much less clear-cut, and
in general optimal domains are only known to exist in the class of
quasi-open sets~\cite{BDM,Henrot}; in
fact, the proof that the corresponding eigenfunctions are
Lipschitz continuous is very recent~\cite{Bu15}. See~\cite{Henrot} for a general account on the subject.

Even though the curl operator plays a preponderant role in different
physical contexts such as fluid mechanics and electromagnetic theory,
the literature about the corresponding optimal domains is surprisingly
scarce. For example, using numerical computations it is easy to show that the ball is not a (locally) optimal
domain. A different but somehow related optimization problem was
considered in~\cite{Cantarella}; there, the stress is on the
Biot--Savart operator, which is an inverse of sorts for the curl
operator that appears in the definition of the helicity. In this
paper, the authors obtain necessary conditions for the existence of
optimal domains for this problem and conjecture that there should not
exist any smooth axisymmetric optimal domains. As we will see later,
there are interesting similarities between this problem and the
question of optimal domains for the curl operator. For related minimization problems in the context of helicity of compactly supported vector field in $\RR^3$ see e.g.~\cite{Freedman,LS}.

Our goal in this paper is to show that there are no $C^{2,\al}$-smooth
axisymmetric
optimal domains for the first positive (or negative) eigenvalue of
curl operator, modulo an additional technical assumption. Of course,
given the symmetries of the problem, axisymmetric domains are a
particularly relevant class of sets to analyze. We stress that $\mu_1(\Om)$ and $-\mu_{-1}(\Om)$ are lower bounded by a constant that only depends on $|\Om|$ as we shall prove in Appendix~\ref{Appendix}, but probably this bound cannot be achieved (at least within the class of smooth enough domains).

Let us start by introducing some notation. Consider cylindrical
coordinates $(z,r,\varphi)\in \RR\times (0,\infty)\times \TT$
on~$\RR^3$, where $\TT:=\RR/(2\pi\ZZ)$. We will henceforth assume that
the axis of symmetry of the domain~$\Om$ is
the $z$-axis, that is, the line $\cZ:=\{r=0\}$. Away
from the axis, the domain
$\Om$ can be written as
\[
\Om\backslash \cZ=\{(z,r,\vp): (z,r)\in D\,,\; \vp\in\TT\}\,.
\]
We will refer to the planar domain $D\subset \RR\times (0,\infty)$ as
the {\em section}\/ of~$\Om$. The distance from a
point~$x$ to the $z$-axis, which is just its coordinate~$r$, will be
denoted by $r(x)$. We denote by
\[
\delta_\Om:=\inf\{r(x) : x\in\Om\}
\]
the distance from the domain~$\Om$ to the $z$-axis. If $\de_\Om=0$, it
means that $\overline\Om$ intersects the $z$-axis. In contrast, if
$\de_\Om>0$, one can write this domain in terms of its section as
$\Om= D\times\TT$, and the closure of~$D$ is contained in the
half-space $\RR\times(0,\infty)$. We will use the notation
\[
\cR_\Om:=\{x\in\partial\Om: r(x)=\delta_\Om\}\,,\qquad
\cR_D:=\{(z,r)\in\pd D: r=\de_\Om\}
\]
for the set of points on the boundary of the domain~$\Om$, or of its
section~$D$, that are closest to the symmetry axis.

\begin{theorem}\label{T.main}
Let $\Om$ be an axisymmetric bounded domain with a $C^{2,\al}$
boundary. If $\Om$ does not intersect the $z$-axis, let us further assume that
the boundary~$\pd\Om$ and the set of innermost boundary
points~$\cR_\Om$ are connected. Then the domain~$\Om$ is not locally optimal for the first positive or
negative curl eigenvalue.
\end{theorem}

Note that the condition that $\cR_\Om$ be connected is generic if $\de_\Om>0$. In
fact, it is easy to check that for any $C^{2,\al}$ axisymmetric domain $\Om$ that does not intersect the $z$-axis,
there is an axisymmetric domain $\Om'$ that is $C^{2,\al}$-close to
$\Om$ such that $\cR_{\Om'}$ consists of a single point. An immediate consequence is that there are no optimal domains in the
quite natural class of axisymmetric domains whose section~$D$ is convex:

\begin{corollary}\label{C.1}
There are no $C^{2,\alpha}$-smooth locally optimal domains for the first positive or
negative curl eigenvalue that are axisymmetric with a convex section. In particular, the ball is not a locally optimal domain.
\end{corollary}

The paper is organized as follows. In Section~\ref{S.domains} we will obtain a necessary condition for a $C^{2,\al}$ bounded domain to be locally optimal, cf.~Proposition~\ref{T.necessary}. In particular, in Corollary~\ref{C.geod} we provide a topological obstruction for a domain (not necessarily axisymmetric) to be locally optimal; this result complements Corollary~\ref{C.1} above. The proof of Theorem~\ref{T.main} is presented in Section~\ref{S.proof}, where we also show that $\mu_1(\Om)$ is simple and the corresponding eigenfield $u_1$ is axisymmetric if $\Om$ is a locally optimal axisymmetric domain with $C^{2,\al}$ connected boundary (see Corollary~\ref{C.simple}). Finally, we include Appendix~\ref{Appendix} where we prove that the first positive and negative curl eigenvalues are lower bounded by a constant that only depends on the volume of the domain; while this is reminiscent of the Faber--Krahn inequality for the Dirichlet Laplacian, it is quite different from it in the sense that the bound we obtain is not sharp and probably it cannot be achieved.

\section{A necessary condition for optimal domains}
\label{S.domains}

In this section we prove that any curl eigenfield of a (locally) optimal domain associated with the first positive (or negative) eigenvalue must have constant pointwise norm on the boundary (the same constant for all the connected components). In turn this will imply that the boundary of the domain consists of tori and that the (unparametrized) integral curves of the eigenfield are geodesics with respect to the induced metric. These results are analogous to those obtained for the helicity maximization problem considered in~\cite[Theorem D]{Cantarella}. We remark that the domain is not assumed to be axisymmetric in this section.

Let us first recall that the curl eigenfields are smooth $u_k\in C^\infty\loc(\Om)$. Indeed, since $u_k$ also satisfies (component-wise)
\[
\De u_k+\mu^2_k u_k=0
\]
in~$\Om$, the result follows by elliptic
regularity (in fact, they are real-analytic in $\Om$). Moreover, since $\partial\Om$ is $C^{2,\alpha}$, it is standard that $u_k$ is $C^{1,\alpha}$ up to the boundary~\cite{JMPA}. The same results hold for harmonic fields. In particular, we conclude that $u_k\restr$ and $h\restr$ belong to $C^{1,\al}(\partial\Omega)$. We will use this boundary regularity property in what follows without further mention.


\begin{proposition}\label{T.necessary}
  If $\Om$ is a $C^{2,\al}$ locally optimal domain for the first
  positive curl eigenvalue, then any
  eigenfield~$u_1$ with eigenvalue $\mu_1(\Om)$ satisfies that its
  pointwise norm on~$\pd\Om$ is constant, i.e., $|u_1|^2\restr=c$ for some $c>0$. The analogous statement holds if $\Om$ is a locally optimal
 domain for the first negative eigenvalue.
\end{proposition}

\begin{proof}
Let $V$ be a smooth bounded vector field on $\RR^3$ which is assumed to be divergence-free in a neighborhood of $\overline\Om$, and let
$\Phi^t$ denote its time-$t$ flow, which is a diffeomorphism of $\RR^3$. Let us now define
\[
\Om^t:=\Phi^t(\Om)\,,\qquad v^t:=\Phi^t_* u_1\,,
\]
where $\Phi^t_* u_1$ denotes the push-forward of the vector
field~$u_1$ along the diffeomorphism~$\Phi^t$ and we normalize the
eigenfunction so that its $L^2(\Om)$ norm is $\|u_1\|=1$. Obviously, $|\Om^t|=|\Om|$ if $|t|<\ep_0$, with $\ep_0>0$ a small enough constant, because $V$ is divergence-free in a neighborhood of $\overline\Om$. Also notice that $v^t$ depends smoothly on $t$.

Let $N^t$ be the outward-point unit normal to the domain $\Om^t$ and let $N^{t\flat}$ denote the 1-form dual to~$N^t$ via the
Euclidean metric. Since the kernel of $N^{t\flat}$ at a point $x\in\pd\Om^t$ is the tangent plane $T_x(\pd\Om^t)$, and the diffeomorphism $\Phi^t$ maps the tangent plane $T_{\Phi^{-t}(x)}(\pd\Om)$ onto $T_x(\pd\Om^t)$, it is immediate that $N^{t\flat}=F_t\Phi^t_*N^\flat$ for some positive function $F_t$ on $\pd\Om^t$.
Then considering the coupling of $N^{t\flat}$ with $v^t$, it is obvious that
\begin{equation}\label{Nv}
N^t\cdot v^t\restrt= N^{t\flat}(v^t\restrt)=F_tN^\flat(u_1\restr)\circ\Phi^{-t}= F_t(N\cdot u_1\restr)\circ\Phi^{-t}=0\,;
\end{equation}
where we have used that $\Phi^t_*N^\flat(\Phi^t_*u_1\restr)=N^\flat(u_1\restr)\circ \Phi^{-1}$ and $N\cdot u_1\restr=0$. Moreover, the
derivative of $v^t$ with respect to~$t$ is given by the Lie derivative
\[
\pd_t v^t= (v^t\cdot\nabla) V-(V\cdot \nabla) v^t\,,
\]
which we can rewrite (if $|t|<\ep_0$) using vector calculus identities and the fact
that $\Div V=0$ in a neighborhood of $\overline \Om$ as
\begin{equation}\label{curlv}
\pd_t v^t= \curl (V\times v^t) -(\Div v^t)\, V\,.
\end{equation}
We can take the divergence in this equation to find a
linear equation for $\Div v^t$:
\[
\pd_t \Div v^t= -(V\cdot \nabla) \Div v^t\,.
\]
Since $\Div v^t=0$ at $t=0$, we infer that
\begin{equation}\label{divvt}
\Div v^t=0
\end{equation}
for all~$|t|<\ep_0$.

Armed with these facts, we can now prove that $v^t\in \cK(\Om^t)$ for
all~$|t|<\ep_0$, which amounts to showing that
\[
\int_{\Om^t} v^t\cdot \nabla\psi\, dx=0\qquad\text{and}\qquad \int_{\Om^t} v^t\cdot h\, dx=0
\]
for all $\psi\in H^1(\Om^t)$ and all $h\in\Harmt$ by the Hodge
decomposition theorem. Let us start with
the first integral, where, by a density argument, one can safely
assume that $\psi\in C^1(\overline {\Om^t})$. As $\Div v^t=0$ by~\eqref{divvt}, we immediately
obtain that
\begin{equation}\label{grad}
\int_{\Om^t} v^t\cdot \nabla\psi\, dx=\int_{\pd\Om^t} N^t\cdot
v^t\restrt\, \psi\, dx=0
\end{equation}
where we have also used~\eqref{Nv}.

To tackle the second integral, let us denote by
$\{h_j^t\}_{j=1}^{b_1}$ a basis of the space~$\Harmt$ depending
smoothly on the parameter~$t$. We recall that the dimension
of~$\Harmt$ is independent of~$t$ and given by the first Betti
number~$b_1$ of the domain~$\Om$. Since $\curl h^t_j=0$ on~$\Om^t$ for
all~$t$, the derivative of $h^t_j$ with
respect to~$t$ also satisfies the equation
\[
\curl \pd_t h^t_j=0
\]
on~$\Om^t$, so one write it as the sum
\begin{equation}\label{dec1}
\pd_t h^t_j= H^t_j + \nabla\psi^t_j
\end{equation}
of a harmonic field
\begin{equation}\label{dec2}
H^t_j=\sum_{k=1}^{b_1} c_{jk}(t)\, h_k^t \in \Harmt
\end{equation}
and the gradient of a scalar
function $\psi_j^t\in H^1(\Om^t)$.

It then follows that the time
derivative of
\[
f_j(t):=\int_{\Om^t} v^t\cdot h^t_j\, dx
\]
is of the form
\begin{align*}
f_j'(t)&= \int_{\Om^t}
         v^t\cdot (\pd_th^t_j)\, dx +\int_{\Om^t} (\pd_tv^t)\cdot h^t_j\, dx + \int_{\Om^t} V\cdot\nabla(v^t\cdot h^t_j)\, dx\,,
\end{align*}
where the last term (which corresponds to the so-called material
derivative) arises from the fact that the domain~$\Om^t$ moves
along the flow of~$V$. The first term can be readily computed
using~\eqref{dec1}-\eqref{dec2}:
\begin{align*}
\int_{\Om^t}
         v^t\cdot (\pd_th^t_j)\, dx&= \sum_{k=1}^{b_1}c_{jk}(t)\int_{\Om^t}
         v^t\cdot h^t_k\, dx + \int_{\Om^t}
         v^t\cdot \nabla\psi_j^t\, dx\\
&= \sum_{k=1}^{b_1}c_{jk}(t)\, f_k(t) \,.
\end{align*}
Here we have used that~$v^t$ is orthogonal to all gradients by
Equation~\eqref{grad}. Using~\eqref{curlv} and the fact that $\Div
v^t=\Div V= 0$, the second and third terms yield
\begin{align*}
I&:=\int_{\Om^t} [(\pd_tv^t)\cdot h^t_j+
V\cdot\nabla(v^t\cdot h^t_j)]\, dx\\
&= \int_{\Om^t}\curl (V\times
                          v^t)\cdot h_j^t\, dx + \int_{\pd\Om^t} (V\cdot
                          N^t)\, (v^t\restrt\cdot h^t_j\restrt)\, dS\,,
\end{align*}
where we know that the boundary term is well defined because
\[
v^t\restrt= \Phi^t_*(u_1\restr)
\]
with $u_1\restr\in C^{1,\al}(\pd\Om)$. Integrating by parts in the first integral and
using Equation~\eqref{Nv} and that $h^t_j$ is curl-free, we
obtain
\begin{align*}
\int_{\Om^t}\curl (V\times
                          v^t)\cdot h_j^t\, dx  = -\int_{\pd\Om^t} (V\cdot
                          N^t)\, (v^t\restrt\cdot h^t_j\restrt)\, dS\,,
\end{align*}
which implies that $I=0$. Hence, putting everything together, for all
$1\leq j\leq b_1$ we have
\[
f'_j(t)=\sum_{k=1}^{b_1}c_{jk}(t)\, f_k(t)\,.
\]
Since $f_k(0)=0$, we infer that $f_k(t)=0$ for all~$t$ and
all~$k$. This completes the proof that $v^t\in \cK(\Om^t)$ for
all~$|t|<\ep_0$.

Since $\curl$ defines a self-adjoint operator on~$\cK(\Om^t)$ with
domain~$\cD_{\Om^t}$, let us denote by $T^t$ its inverse, which is a
compact operator on~$\cK(\Om^t)$. Let us denote by $u_k^t$ and~$\mu_k^t$ ($k\in\ZZ\backslash\{0\}$)
an orthonormal basis of eigenfunctions of the curl operator
on~$\cD_{\Om^t}$ and the corresponding eigenvalues. Expanding an arbitrary vector field
$w\in\cK(\Om^t)$ as
\[
w=\sum_{k\in\ZZ\backslash\{0\}} w_k\, u_k^t\,,
\]
with $w_k\in\RR$, it follows that
\[
\langle T^tw,w\rangle =\sum_{k\in\ZZ\backslash\{0\}}
\frac{w_k^2}{\mu_k^t}\,,\qquad
\|w\|^2=\sum_{k\in\ZZ\backslash\{0\}} w_k^2\,,
\]
where the inner product and the norm are obviously those of~$L^2(\Om^t)$.
Hence for all~$w\in \cK(\Om^t)$,
\[
\frac1{\mu_{-1}^t}\leq \frac{\langle T^tw,w\rangle}{\|w\|^2}\leq \frac1{\mu_1^t}\,,
\]
and these inequalities are sharp because they are saturated for
$w=u_1^t$ or $w=u_{-1}^t$.

It follows from the above argument that if~$\Om$ is a locally optimal domain
for the first positive curl eigenvalue, then the time derivative of
the function
\[
R(t):=\frac{\langle T^tv^t,v^t\rangle}{\|v^t\|^2}
\]
must satisfy $R'(0)=0$. Notice that $R(0)=1/\mu_1$. Although it is not obvious a priori, in the following computations we shall prove that $R(t)$ is smooth in $t$, which justifies to take the time derivative at $t=0$.

First, using Equation~\eqref{curlv}, the time
derivative at~0 of the denominator (which is clearly smooth in $t$ because $v^t$ is) is readily shown to be
\begin{align}
\pd_t|_{t=0} \|v^t\|^2&= 2\langle u_1,\pd_t v^t|_{t=0}\rangle\notag+\int_{\pd\Om}(V\cdot N)|u_1|^2\restr \\
&= 2\int_\Om u_1\cdot \curl (V\times u_1)\, dx\notag+\int_{\pd\Om}(V\cdot N)|u_1|^2\restr \\
&=2\int_\Om \curl u_1\cdot (V\times u_1)\, dx - \int_{\pd\Om}(V\cdot
  N)\, |u_1|^2\restr\, dS\notag\\
&=- \int_{\pd\Om}(V\cdot
  N)\, |u_1|^2\restr\, dS\,,\label{denom}
\end{align}
where we have used that $\curl u_1=\mu_1u_1$.

The computation of the time derivative of the numerator is more
involved. We start by defining the $L^2$-valued 1-form associated to
$T^tv^t$ via the Euclidean metric, which we denote
by~$\al^t$. Notice that
\begin{align}
d(\al^t-\Phi^t_* \al^0)&=d\al^t-\Phi^t_* d\al^0\notag\\
&= i_{\Phi^t_* u_1}(dx_1\wedge dx_2\wedge dx_3)- \Phi^t_*
  i_{u_1}(dx_1\wedge dx_2\wedge dx_3)\notag\\
&= i_{\Phi^t_* u_1}(dx_1\wedge dx_2\wedge dx_3)-
  i_{\Phi^t_* u_1} [\Phi^t_* (dx_1\wedge dx_2\wedge dx_3)]\notag \\
&=0\,,\label{cerrada}
\end{align}
where $\wedge$, $d$ and $i_W$ respectively denote the exterior
product, the differential and the contraction with the vector
field~$W$. We
have also employed that the push-forward commutes with the exterior
derivative and the volume 3-form $dx_1\wedge dx_2\wedge dx_3$ is
invariant under the flow $\Phi^t$ (that is, $\Phi^t_* (dx_1\wedge
dx_2\wedge dx_3) = dx_1\wedge dx_2\wedge dx_3$ for $|t|<\ep_0$) because~$V$ is divergence-free in a neighborhood of $\overline\Om$. Furthermore,
we have applied to~$Y:=T^tv^t$ the well-known formula
\[
i_{\curl Y}(dx_1\wedge dx_2\wedge dx_3)= d\be\,,
\]
where $\be$ is the 1-form dual to the vector field~$Y$, and used that $\curl T^tv^t=v^t$.

We can now write the numerator of the function~$R(t)$ as
\begin{align*}
g(t):=\langle T^tv^t,v^t\rangle = \int_{\Om^t} \al^t\wedge d\al^t\,.
\end{align*}
By the Hodge
decomposition theorem, Equation~\eqref{cerrada} implies that
there is a 1-form $\be^t$, dual to a vector field $h^t\in\Harmt$, and
a function $\psi^t\in H^1(\Om^t)$ such that
\[
\al^t= \Phi^t_*\al^0 + \be^t + d\psi^t\,.
\]
One then has
\begin{align*}
g(t)=\int_{\Om^t} \Phi^t_*\al^0\wedge d\al^t + \int_{\Om^t}
  \be^t\wedge d\al^t+ + \int_{\Om^t} d\psi^t\wedge d\al^t \,.
\end{align*}
The second term is very easy to compute, since by the definition of
the various 1-forms one has
\begin{align*}
  \int_{\Om^t}  \be^t\wedge d\al^t= \int_{\Om^t} h^t\cdot \curl (T^t
                                    v^t)\, dx = \int_{\Om^t} h^t\cdot
  v^t\, dx=0
\end{align*}
because $v^t\in \cK(\Om^t)$, and the third term is similar:
\begin{align*}
  \int_{\Om^t}  d\psi^t\wedge d\al^t= \int_{\Om^t} \nabla\psi^t\cdot \curl (T^t
                                    v^t)\, dx = \int_{\Om^t} \nabla\psi^t\cdot
  v^t\, dx=0\,.
\end{align*}
Hence $g(t)$ coincides with the first summand, which can be rewritten
using~\eqref{cerrada} as
\begin{align*}
g(t)&= \int_{\Om^t} \Phi^t_*\al^0\wedge d(\Phi^t_*\al^0)\\
&=
      \int_{\Om^t} \Phi^t_*(\al^0\wedge d\al^0)\\
&= \int_\Om \al^0\wedge d\al^0\\
&=\int_\Om u_1\cdot T^0u_1\, dx\\
&=\frac1{\mu_1}\,.
\end{align*}

This shows that
\[
R(t)=\frac1{\mu_1\|v^t\|^2}\,,
\]
and in particular, $R(t)$ is smooth in $t$. The identity~\eqref{denom} and the fact that $\|u_1\|=1$ readily yields
\[
R'(0)=\frac 1{\mu_1} \int_{\pd\Om}(V\cdot
  N)\, |u_1|^2\restr\, dS\,.
\]
It is standard that this integral vanishes for any divergence-free
vector field $V\in C^\infty(\BOm)$ if and only if $|u_1|^2\restr$ is a constant $c$, the same on each connected component of~$\pd\Om$.

Finally, assume that $c=0$. It is then easy to check that the vector field
\[
u(x):=\begin{cases}
u_1(x) &\text{ if }\; x\in\Om\,,\\
0 &\text{ if }\; x\not\in\Om\,,
\end{cases}
\]
is in $H^1(\RR^3)$ and satisfies the equation $\curl u=\mu_1(\Om)u$ in $\RR^3$ in the sense of distributions. The Liouville theorem for Beltrami flows~\cite{Nadirashvili,CC} then implies that $u=0$ in $\RR^3$, which is a contradiction, so we conclude that $c>0$. Analogous arguments work for the first negative curl eigenvalue, and the proposition then follows.
\end{proof}

\begin{remark}
Although we will only consider $C^2$ domains in this article, Proposition~\ref{T.necessary} also holds for Lipschitz optimal domains, where all the boundary restrictions have to be understood in the sense of traces.
\end{remark}

\begin{corollary}\label{C.geod}
If $\Om$ is a $C^{2,\alpha}$ locally optimal domain for the first
positive curl eigenvalue then each connected component of $\pd\Om$ is diffeomorphic to $\mathbb T^2$ and the (unparameterized) integral curves of $u_1\restr$ are geodesics with respect to the induced metric. The analogous statement holds for the first negative eigenvalue. In particular, there are no $C^{2,\alpha}$-smooth locally optimal domains that are convex.
\end{corollary}
\begin{proof}
Let $\Si_k$ be a connected component of $\pd\Om$. If the Euler characteristic $\chi(\Si_k)\neq 0$, then $u_1{\!\upharpoonright_{\Si_k}}$ vanishes at some point by the Poincar\'e--Hopf index theorem. In view of Proposition~\ref{T.necessary} we conclude that $c=|u_1|^2\restr=0$, which is a contradiction, so we deduce that $\chi(\Si_k)=0$ for all the connected components of $\pd\Om$, which means that they are diffeomorphic to $\mathbb T^2$. Now, the well-known identity
\[
\nabla_uu=\frac12\nabla|u|^2-u\times \curl u
\]
implies that $\nabla_{u_1} u_1=\frac12\nabla|u_1|^2$. Since $u_1$ is $C^{1,\alpha}$ up to the boundary and $|u_1|^2\restr=c$, the restriction $v:=u_1\restr$ on $\pd\Om$ satisfies $\nabla_v v=0$, where $\nabla_v$ is the covariant derivative along $v$ with respect to the induced metric. This is equivalent to saying that the (unparametrized) integral curves of $v$ are geodesics. The last claim in the statement follows from the fact that the boundary of any $C^{2,\al}$ convex domain is diffeomorphic to a sphere.
\end{proof}

\section{Proof of Theorem~\ref{T.main}}
\label{S.proof}

Let $\Om$ be an axisymmetric $C^{2,\alpha}$ locally optimal domain for the first positive curl eigenvalue and consider an eigenfield $u_1$ as in Proposition~\ref{T.necessary}. Corollary~\ref{C.geod} implies that all the connected components of $\pd\Om$ are diffeomorphic to $\mathbb T^2$, and hence $\de_\Om>0$, i.e., $\Om$ does not intersect the $z$-axis. The domain $\Om$ is then of the form $D\times \TT$ where $D\subset\RR\times (0,\infty)$ is the section of $\Om$. We assume in what follows that $\pd\Om$ (and so $\pd D$) is connected.

Assume that the eigenfield $u_1$ is axisymmetric. It can then be written in cylindrical coordinates (in terms of the orthonormal basis $\{e_z,e_r,e_\vp\}$) as
\begin{equation}\label{defu}
u_1= \frac1r\left[ \pd_r\psi\, e_z-\pd_z\psi\, e_r +\mu_1\psi\, e_\vp\right]\,,
\end{equation}
where $\mu_1\equiv \mu_1(\Om)$ is the first positive curl eigenvalue, and the function $\psi(z,r)$ satisfies the Grad-Shafranov equation
\begin{equation}\label{elliptic}
L\psi=-\mu_1^2\psi
\end{equation}
in the section $D$. In particular, $\psi$ belongs to $C^{2,\al}(\overline D)$. Here we have set
\begin{equation}\label{defL}
L\psi:=\pd_{zz}\psi+\pd_{rr}\psi - \frac1r\pd_r\psi\,.
\end{equation}
Since $u_1$ is tangent to $\pd\Om$ and $|u_1|^2\restr$ is constant, it follows that $\psi$ satisfies the following boundary conditions
\begin{align}
  \psi|_{\pd D}=c_1\,,\label{Dirichlet}\\
   \left.\frac{(\nabla\psi)^2+ \mu_1^2c_1^2}{r^2}\right|_{\pd D}=c_2\,,\label{Neumann}
\end{align}
for some constants $c_1$ and $c_2>0$.

Since $\pd D$ is connected, $D$ is a simply connected planar domain, so $\Om$ is diffeomorphic to a solid torus. The space of harmonic fields $\cH_\Om$ then has dimension one, and it is trivial to check that the only harmonic field (up to a constant factor) is given by
\begin{equation}\label{eq.harm}
h=\frac1r\,e_\vp\,.
\end{equation}
The fact that $u_1\in \cK(\Om)$ implies that
\[
0=\int_\Om u_1\cdot h dx=2\pi\mu_1\int_{D}\frac{\psi(z,r)}{r}\,dzdr\,.
\]
Using Equation~\eqref{elliptic}, this yields
\begin{equation}\label{eq.harm}
0=\int_{D}\frac{L\psi}{r}\,dzdr=\int_{D}\Big[\partial_z\Big(\frac{\partial_z\psi}{r}\Big)+\partial_r\Big(\frac{\partial_r\psi}{r}\Big)\Big]\,dzdr=
\int_{\pd D}\frac{\nabla\psi\cdot N}{r}dS\,.
\end{equation}
Here $\nabla\psi:=\partial_z\psi e_z+\pd_r\psi e_r$ and $N$ is the outward-point unit normal to $\pd D$. To pass to the last equality we have simply integrated by parts.

If $c_1=0$ in Equation~\eqref{Dirichlet}, then~\eqref{Neumann} and~\eqref{eq.harm} imply that $c_2=0$, which means that $|u_1|^2\restr=0$, which contradicts Proposition~\ref{T.necessary}. Let us now consider the case $c_1\neq 0$. The connectedness of $\pd D$ and the fact that $\nabla\psi\restrr$ cannot be identically zero, imply that Equation~\eqref{eq.harm} can be fulfilled only if the zero set of $\nabla\psi$ on $\pd D$ is nonempty and consists of at least two connected components. Let us characterize the zero set $Z$ of $\nabla\psi\restrr$. Take a point $(r_*,z_*)\in Z$ and assume that there is a point $(r_0,z_0)\in\pd D$ with $r_0<r_*$; then Equation~\eqref{Neumann} implies that
\[
\frac{\mu_1^2c_1^2}{r_*^2}=c_2\,,
\]
and
\[
c_2=\frac{(\nabla \psi)^2|_{(r_0,z_0)}+\mu_1^2c_1^2}{r_0^2}>\frac{(\nabla \psi)^2|_{(r_0,z_0)}+\mu_1^2c_1^2}{r_*^2}=\frac{(\nabla \psi)^2|_{(r_0,z_0)}}{r_*^2}+c_2\,,
\]
which is a contradiction. We hence conclude that $Z\subset\cR_{D}$, i.e. the set of points on $\pd D$ that are closest to the $z$-axis, and being $Z$ nonempty, the inclusion $\cR_{D}\subset Z$ also follows from Equation~\eqref{Neumann}. This shows that $\cR_{D}$ and so $\cR_\Om$ consists of at least two connected components if $\Om$ is a locally optimal domain, thus proving the second claim in Theorem~\ref{T.main} provided that $u_1$ is axisymmetric.

In general, an eigenfield $u_1$ associated to $\mu_1$ does not need to be axisymmetric and reads in cylindrical coordinates as
\[
u_1= u_{1z}\, e_z+u_{1r}\, e_r +u_{1\vp}\, e_\vp\,,
\]
where $u_{1z},u_{1r},u_{1\vp}$ are functions of $(z,r,\vp)\in\Om$. Let us now define the axisymmetric vector field
\[
u_1^S:=Su_{1z}\, e_z+Su_{1r}\, e_r +Su_{1\vp}\, e_\vp\,,
\]
where the axisymmetrization operator $S:C^{k,\al}(\Om)\to C^{k,\al}(\Om)$ is given by
\[
Sf(z,r):=\frac1{2\pi}\int_0^{2\pi}f(z,r,\vp)\,d\vp\,.
\]
Since $\Om$ is axisymmetric, it is easy to check that $\curl u_1^S=\mu_1 u_1^S$ and $u_1^S\in \cK(\Om)$, so $u_1^S$ is also an eigenfield of curl with eigenvalue $\mu_1$. We claim that $u_1^S$ is not identically zero on $\Om$.

Indeed, assume that there is a point $p_0=(z_0,r_0,\vp_0)\in\pd\Om$ such that $u_{1\vp}(p_0)=0$. Then $u_1\restr(p_0)$ is tangent to the meridian of $\pd\Om$ passing by the point $p_0$. Noticing that the (unparametrized) integral curves of $u_1\restr$ are geodesics on $\pd\Om$ by Corollary~\ref{C.geod}, and the fact that the meridian circles of an axisymmetric surface are geodesics, it immediately follows that the meridian $\ga_0$ on $\pd\Om$ corresponding to the point $p_0$ is an (unparametrized) integral curve of $u_1\restr$. Since $|u_1|^2\restr=c>0$, calling $D_0$ the disk $\{\vp=\vp_0\}$ in $\Om$ bounded by $\ga_0$, we can write
\begin{align*}
0\neq& \int_{\ga_0}u_1\,dl=\int_{D_0}\curl u_1\cdot N\,dS=\mu_1\int_{D_0} u_1\cdot N dzdr=\frac{\mu_1}{2\pi}\int_\Om u_1\cdot Ndzdrd\vp
\\&=\frac{\mu_1}{2\pi}\int_{\Om}u_{1\vp}dzdr=\frac{\mu_1}{2\pi}\int_{\Om}u_1\cdot h dx=0\,.
\end{align*}
Here $N=e_\vp$ is a unit vector normal to $D_0$ and $h$ is the unique harmonic field, cf.~Equation~\eqref{eq.harm}, in $\Om$. In the first equality we have used Stokes theorem and to pass from the integral on $D_0$ to an integral on $\Om$ we have used that $u_1$ is divergence-free and hence its flux through any disk $\{\vp=\vp_0\}$ does not depend on the angle. Since this equation yields a contradiction, we conclude that the component $u_{1\vp}$ does not vanish at any point of $\pd\Om$. Accordingly, the axisymmetric vector field $u_1^S$ cannot be identically zero, as claimed.

Summarizing, we have proved that for any optimal axisymmetric domain there exists a nontrivial axisymmetric curl eigenfield $u_1^S$ associated with the first positive curl eigenvalue $\mu_1$. The theorem then follows applying the previous discussion to the field $u_1^S$. The case of the first negative curl eigenvalue is completely analogous.

We conclude this section with a corollary that establishes that the first positive (or negative) curl eigenvalue of a locally optimal axisymmetric domain is simple. Notice that this is a very special property of optimal domains, because the first curl eigenvalue of a general bounded domain does not need to be simple (in contrast with the case of the Dirichlet Laplacian).

\begin{corollary}\label{C.simple}
Let $\Om$ be an axisymmetric bounded domain with $C^{2,\al}$ connected boundary. If it is locally optimal for the first positive curl eigenvalue $\mu_1(\Om)$, then this eigenvalue is simple and the corresponding curl eigenfield $u_1$ is axisymmetric. The same result holds if the domain is locally optimal for the first negative curl eigenvalue.
\end{corollary}
\begin{proof}
Since $\Om$ is locally optimal and $C^{2,\al}$, Theorem~\ref{T.main} implies that $\Om=D\times\TT$ for some section $D$ whose closure is contained in $\RR\times(0,\infty)$. Let $u_1$ and $v_1$ be two linearly independent eigenfields associated to $\mu_1$ and consider their axisymmetrizations $S u_1$ and $S v_1$ as defined above. Recall that we have proved that the axisymmetrization of any curl eigenfield of $\mu_1$ is another nontrivial curl eigenfield. It follows from Proposition~\ref{T.necessary} that any linear combination $aSu_1+bSv_1$, $a,b\in\RR$, has constant pointwise norm on $\pd\Om$, i.e.
$$|aSu_1+bSv_1|^2\restr=c(a,b)\,,$$
a constant that may depend on $a$ and $b$. It is then easy to check that the angle $\Theta$ formed by the vectors $Su_1(x)$ and $Sv_1(x)$ does not depend on the point $x\in\pd\Om$.

The fields $Su_1$ and $Sv_1$ can be written as in Equation~\eqref{defu} for some functions $\psi$ and $\tilde\psi$ on $D$, respectively. Since $\pd\Om$ is connected, the proof of Theorem~\ref{T.main} also shows that $\nabla\psi\restrr$ and $\nabla\tilde\psi\restrr$ vanish exactly on the same set $\cR_D$. In view of Equation~\eqref{defu}, the fields $Su_1\restr$ and $Sv_1\restr$ are then collinear at any point $p\in\cR_\Om$ (the are tangent to the rotation field $e_\varphi$), so from the argument above we conclude that they are collinear everywhere on $\pd\Om$. The fact that they have constant pointwise norm on $\pd\Om$ hence implies that there are constants $a_0,b_0$ such that the curl eigenfield $a_0Su_1+b_0Sv_1$ satisfies
\[
(a_0Su_1+b_0Sv_1)\restr =0\,.
\]
In view of Proposition~\ref{T.necessary}, it follows that $a_0Su_1+b_0Sv_1=0$ in $\Om$, which means that the axisymmetrization of the curl eigenfield $a_0u_1+b_0v_1$ is identically zero. Since we proved above that this cannot happen unless the linear combination $a_0u_1+b_0v_1$ is zero itself, we finally conclude that the eigenvalue $\mu_1$ is simple. The same axisymmetrization argument also shows that the corresponding eigenfield $u_1$ must be axisymmetric.
\end{proof}

\section*{Acknowledgements}

A.E.\ is supported by the ERC Starting Grant~633152. D.P.-S.
is supported by the grants MTM PID2019-106715GB-C21 (MICINN) and Europa Excelencia EUR2019-103821 (MCIU). This work is supported in part by the ICMAT--Severo Ochoa grant SEV-2015-0554 and the CSIC grant 20205CEX001.

\appendix

\section{A uniform lower bound for the first curl eigenvalue}
\label{Appendix}

In this section we show that both the first positive and negative curl eigenvalues are lower bounded by a constant that only depends on the volumen of the domain. However, this is rather different from the Faber--Krahn inequality for the Dirichlet Laplacian because here the bound is far from sharp, and probably it cannot be achieved. A similar bound (but upper instead of lower) holds for the helicity maximization problem considered in~\cite[Theorem E]{Cantarella}.

\begin{theorem}\label{T.FK}
For any bounded $C^2$ domain $\Om\subset\RR^3$,
\[
\min\{\mu_1(\Om),-\mu_{-1}(\Om)\} \geq \bigg(\frac{4\pi}{3|\Om|}\bigg)^{1/3}\,.
\]
\end{theorem}

\begin{proof}
Since the curl operator is self-adjoint on the domain~$\Dom$, let us denote by $\curl ^{-1}$ the compact self-adjoint operator
on~$\cK(\Om)$ defined by its inverse. It is then clear that
\begin{equation}\label{app}
\frac{\|v\|^2}{\mu_{-1}}\leq \langle \curl ^{-1} v,v\rangle\leq \frac{\|v\|^2}{\mu_1}
\end{equation}
for all $v\in \cK(\Om)$, and that these inequalities are in fact
equalities when $v$ is $u_{-1}$ or $u_1$, respectively.

Given $v\in \Dom$, let us consider the vector field defined by~$v$ via the Biot--Savart integral
\[
\BS v(x):=\int_{\Om} \frac{v(y)\times (x-y)}{4\pi|x-y|^3}\, dy\,.
\]
It is standard (see e.g.~\cite{JMPA}) that, as $N\cdot v\restr = 0$,
\[
\Div \BS v=0\,,\qquad \curl \BS v=v
\]
in~$\Om$. Since $\BS v-\curl^{-1} v$ is curl-free, the Hodge decomposition theorem then
implies
\[
\curl^{-1}v = \BS v+ h_v+ \nabla \vp_v\,,
\]
where $h_v\in\Harm$ and $\vp_v$ is a scalar function in~$H^1(\Om)$.

Using this formula, we obtain that
\begin{align}
\langle \curl ^{-1} v,v\rangle = \langle \BS v,v\rangle + \langle
                                 h_v,v\rangle + \langle
  \nabla\vp_v,v\rangle = \langle \BS v,v\rangle\,,\label{curl}
\end{align}
where we have used that the other two terms vanish because a vector
field $v\in\Dom$ is orthogonal to the kernel of curl. Notice now
that
\begin{align}
|\BS v(x)| &= \bigg|\int_{\Om} \frac{v(y)\times (x-y)}{4\pi|x-y|^3}\,
             dy\bigg|\notag\\
& \leq \frac1{4\pi} \int_{\Om} \frac{|v(y)|}{|x-y|^2}\, dy\notag\\
&\leq \frac1{4\pi} \bigg(\int_{\Om}\frac{|v(y)|^2}{|x-y|^2}\,
  dy\bigg)^{1/2}
  \bigg(\int_{\Om}\frac{dy}{|x-y|^2}\bigg)^{1/2}\,.\label{BS}
\end{align}
Denoting by $\Om^*$ the ball centered at the origin whose volume
equals that of~$\Om$, the rearrangement inequality ensures that
\[
\sup_{x\in\Om}\int_{\Om}\frac{dy}{|x-y|^2}\leq
\sup_{z\in\Om^*}\int_{\Om^*}\frac{dy}{|z-y|^2}\, dy=
\int_{\Om^*}\frac{dy}{|y|^2}=  (48\pi^2|\Om|)^{1/3}\,.
\]
This estimate implies that
\begin{align*}
  \int_{\Om}\int_{\Om}\frac{|v(y)|^2}{|x-y|^2}\, dx\, dy & \leq \bigg(
                                                           \sup_{Y\in\Om}\int_{\Om}\frac{dx}{|x-Y|^2}\bigg)\bigg(
                                                           \int_{\Om}|v(y)|^2\,
                                                           dy\bigg)\\
&\leq  (48\pi^2|\Om|)^{1/3}\|v\|^2\,,
\end{align*}
so now we can go back to the inequality~\eqref{BS}, square it and integrate in $\Om$ to
estimate the $L^2$~norm of $\BS v$ as
\[
\|\BS v\|\leq  \bigg(\frac{3|\Om|}{4\pi}\bigg)^{1/3}\|v\|\,.
\]
By~\eqref{curl}, this yields
\[
|\langle \curl ^{-1} v,v\rangle|\leq  \bigg(\frac{3|\Om|}{4\pi}\bigg)^{1/3}\|v\|^2
\]
for all $v\in\Dom$. Since $\Dom$ is dense in $\cK(\Om)$ and
$\curl^{-1}$ is a bounded linear operator, it then follows that the
estimate holds for all $v\in \cK(\Om)$. In turn, this implies the eigenvalue estimate presented in the
statement of the theorem because the inequalities~\eqref{app} are
saturated when $v=u_1$ or $v=u_{-1}$, in each case.
\end{proof}

\bibliographystyle{amsplain}

\end{document}